\documentclass[english]{article}
\usepackage[T1]{fontenc}
\usepackage[latin9]{inputenc}
\usepackage{color}
\usepackage{babel}
\usepackage{verbatim}
\usepackage{units}
\usepackage{mathtools}
\usepackage{enumitem}
\usepackage{amsmath}
\usepackage{amsthm}
\usepackage{amssymb}
\usepackage[unicode=true]
 {hyperref}

\makeatletter
 \theoremstyle{definition}
 \newtheorem*{defn*}{\protect\definitionname}
 \newcommand\thmsname{\protect\theoremname}
 \newcommand\nm@thmtype{theorem}
 \theoremstyle{plain}
 
 \newenvironment{namedthm}[1][Undefined Theorem Name]{
   \ifx{#1}{Undefined Theorem Name}\renewcommand\nm@thmtype{theorem*}
   \else\renewcommand\thmsname{#1}\renewcommand\nm@thmtype{namedtheorem}
   \fi
   \begin{\nm@thmtype}}
   {\end{\nm@thmtype}}
\theoremstyle{plain}
\newtheorem{thm}{\protect\theoremname}
  \theoremstyle{remark}
  \newtheorem*{rem*}{\protect\remarkname}
  \theoremstyle{plain}
  \newtheorem{cor}{\protect\corollaryname}
  \theoremstyle{plain}
  \newtheorem{lem}{\protect\lemmaname}
  \theoremstyle{plain}
  \newtheorem{prop}{\protect\propositionname}
  \theoremstyle{definition}
  \newtheorem{problem}{\protect\problemname}

\usepackage{dsfont}
\usepackage{mathtools}
\usepackage{enumitem}

\makeatother

  \providecommand{\definitionname}{Definition}
  \providecommand{\lemmaname}{Lemma}
  \providecommand{\problemname}{Problem}
  \providecommand{\propositionname}{Proposition}
  \providecommand{\remarkname}{Remark}
  \providecommand{\theoremname}{Theorem}
\providecommand{\corollaryname}{Corollary}
\providecommand{\theoremname}{Theorem}

\begin{document}

\title{On Exceptional Sets in the Metric Poissonian Pair Correlations problem}

\author{Thomas Lachmann\thanks{The first author is supported by the Austrian Science Fund (FWF):
Y-901.}, and Niclas Technau\thanks{The second author is supported by the Austrian Science Fund (FWF)
projects: W1230, and (for part of the time) by Y-901.}}
\maketitle
\begin{abstract}
Let $\left(a_{n}\right)_{n}$ be a strictly increasing sequence of
positive integers, denote by $A_{N}=\left\{ a_{n}:\,n\leq N\right\} $
its truncations, and let $\alpha\in\left[0,1\right]$. We prove that
if the additive energy $E\left(A_{N}\right)$ of $A_{N}$ is in $\Omega\left(N^{3}\right)$,
then the sequence $\left(\left\langle \alpha a_{n}\right\rangle \right)_{n}$
of fractional parts of $\alpha a_{n}$ does not have Poissonian pair
correlations (PPC) for almost every $\alpha$ in the sense of Lebesgue
measure. Conversely, it is known that $E\left(A_{N}\right)=\mathcal{O}\left(N^{3-\varepsilon}\right)$,
for some fixed $\varepsilon>0$, implies that $\left(\left\langle \alpha a_{n}\right\rangle \right)_{n}$
has PPC for almost every $\alpha$. This note makes a contribution
to investigating the energy threshold for $E\left(A_{N}\right)$ to
imply this metric distribution property. We establish, in particular,
that there exist sequences $\left(a_{n}\right)_{n}$ with 
\[
E\left(A_{N}\right)=\Theta\left(\frac{N^{3}}{\log\left(N\right)\log\left(\log N\right)}\right)
\]
such that the set of $\alpha$ for which $\left(\alpha a_{n}\right)_{n}$
does not have PPC is of full Lebesgue measure. Moreover, we show that
for any fixed $\varepsilon>0$ there are sequences $\left(a_{n}\right)_{n}$
with $E\left(A_{N}\right)=\Theta\left(\frac{N^{3}}{\log\left(N\right)\left(\log\log N\right)^{1+\varepsilon}}\right)$
satisfying that the set of $\alpha$ for which the sequence $\left(\bigl\langle\alpha a_{n}\bigr\rangle\right)_{n}$
does not have PPC is of full Hausdorff dimension.
\end{abstract}
\global\long\def\NPPC{\mathrm{NPPC}}
\global\long\def\FJ{F_{j}}

\section{Introduction}

The theory of uniform distribution modulo $1$ dates back, at least,
to the seminal paper of Weyl \cite{Weyl: =0000DCber die Gleichverteilung von Zahlen mod. Eins}.
Weyl showed, inter alia, that for any fixed irrational $\alpha\in\mathbb{R}$
and integer $d\geq1$ the sequences $\left(\bigl\langle\alpha n^{d}\bigr\rangle\right)_{n}$
are uniformly distributed modulo $1$. However, in recent years various
authors \cite{Aistleitner Larcher Lewko: Additive Energy and the Hausdorff Dimension of the Exceptional Set in Metric Pair Correlation Problems,Heath-Brown: Pair correlation for fractional parts of n^2,Rudnick Sarnak: The pair correlation function of fractional parts of polynomials,Rudnick Sarnak Zaharescu: The distribution of spacings between the fractional parts of n^2,Rudnick Zaharescu: A metric result on the pair correlation of fractional parts of sequences,Truelsen: Divisor problems and the pair correlation for the fractional parts of n^2,Walker: The Primes are not Metric Poissonian}
have been investigating a more subtle distribution property of such
sequences - namely, whether the asymptotic distribution of the pair
correlations has a property which is called Poissonian, and defined
as follows:
\begin{defn*}
Let $\left\Vert \cdot\right\Vert $ denote the distance to the nearest
integer. A sequence $\left(\theta_{n}\right)_{n}$ in $\left[0,1\right]$
is said to have (asymptotically) Poissonian pair correlations, if
for each $s\geq0$ the pair correlation function\footnote{The subscript $2$ in $R_{2}$ indicates that relations of second
order, i.e. pair correlations, are counted.}
\begin{equation}
R_{2}\left(\left[-s,s\right],\left(\theta_{n}\right)_{n},N\right)\coloneqq\frac{1}{N}\#\left\{ 1\leq i\neq j\leq N:\,\left\Vert \theta_{i}-\theta_{j}\right\Vert \leq\frac{s}{N}\right\} \label{eq: definition of the Pair Correlation Counting function}
\end{equation}
tends to $2s$ as $N\rightarrow\infty$. Moreover, let $\left(a_{n}\right)_{n}$
denote a strictly increasing sequence of positive integers. If no
confusion can arise, we write
\[
R\left(\left[-s,s\right],\alpha,N\right)\coloneqq R_{2}\left(\left[-s,s\right],\left(\alpha a_{n}\right)_{n},N\right)
\]
and say that a sequence $\left(a_{n}\right)_{n}$ has metric Poissonian
pair correlations if $\left(\alpha a_{n}\right)_{n}$ has Poissonian
pair correlations for almost all $\alpha\in\left[0,1\right]$ in the
sense of Lebesgue measure.
\end{defn*}
It is known that if a sequence $\left(\theta_{n}\right)_{n}$ has
Poissonian pair correlations, then it is uniformly distributed modulo
$1$, cf. \cite{Aistleitner Lachmann Pausinger: Pair correlations and equidistribution,Larcher Grepstad: On pair correlation and discrepancy}.
Yet, the sequences $\left(\left\langle \alpha n^{d}\right\rangle \right)_{n}$
do \emph{not} have Poissonian pair correlations for \emph{any} $\alpha\in\mathbb{R}$
if $d=1$. For $d\geq2$, Rudnick and Sarnak \cite{Rudnick Sarnak: The pair correlation function of fractional parts of polynomials}
proved that $\left(n^{d}\right)_{n}$ has metric Poissonian pair correlations
(metric PPC). For alternative proofs, we refer the reader to Heath-Brown
\cite{Heath-Brown: Pair correlation for fractional parts of n^2}
and the work of Marklof and Strömbergsson \cite{Marklof Str=0000F6mbergsson: Equidistribution of Kronecker sequences along closed horocycles}.\footnote{It is worthwhile to mention that the case $d=2$ is of particular
interest for its connection to mathematical physics, see \cite{Rudnick Sarnak: The pair correlation function of fractional parts of polynomials}
for further references.} Given these results, it is natural to investigate which properties
of a sequence of integers $\left(a_{n}\right)_{n}$ implies the metric
PPC of $\left(a_{n}\right)_{n}$. Partial answers are known, e.g.
it follows from work of Boca and Zaharescu \cite{Boca Zaharescu: Pair correlation of values of rational functions (mod p)}
that $\left(P\left(n\right)\right)_{n}$ has metric PPC if $P$ is
any polynomial with integer coefficients of degree at least two. An
interesting general result in this direction is due to Aistleitner,
Larcher, and Lewko \cite{Aistleitner Larcher Lewko: Additive Energy and the Hausdorff Dimension of the Exceptional Set in Metric Pair Correlation Problems}
who used a Fourier analytic approach combined with a bound on GCD
sums of Bondarenko and Seip \cite{Bondarenko Seip: GCD sums and complete sets of square-free numbers}
to relate the metric PPC of $\left(a_{n}\right)_{n}$ with its combinatoric
properties. For stating it, let $\left(a_{n}\right)_{n}$ denote henceforth
a strictly increasing sequence of positive integers and denote the
set of the first $N$ elements of $\left(a_{n}\right)_{n}$ by $A_{N}$.
Moreover, define the additive energy $E\left(I\right)$ of a finite
set integers $I$ via
\[
E\left(I\right)\coloneqq\sum_{\underset{a+b=c+d}{a,b,c,d\in I}}1.
\]
In the following, let $\mathcal{O}$ and $o$ denote the standard
Landau symbols/O-notation.\\
\\
A main finding of \cite{Aistleitner Larcher Lewko: Additive Energy and the Hausdorff Dimension of the Exceptional Set in Metric Pair Correlation Problems}
is the implication that if the truncations $A_{N}$ satisfy 
\begin{equation}
E\left(A_{N}\right)=\mathcal{O}\left(N^{3-\varepsilon}\right)\label{eq: Aistleitner bound}
\end{equation}
for some fixed $\varepsilon>0$, then $\left(a_{n}\right)_{n}$ has
metric PPC. Note that $\left(\#I\right)^{2}\leq E\left(I\right)\leq\left(\#I\right)^{3}$
where $\#I$ denotes the cardinality of $I\subset\mathbb{Z}$. Roughly
speaking, a set $I$ has large additive energy if and only if it contains
a ``large'' arithmetic progression like structure. Indeed, if $\left(a_{n}\right)_{n}$
is a geometric progression or of the form $\left(n^{d}\right)_{n}$
for $d\geq2,$ then (\ref{eq: Aistleitner bound}) is satisfied. Furthermore,
note that the metric PPC property may be seen as a sort of pseudorandomness;
in fact, for a given sequence of $\left[0,1\right]$-uniformly distributed,
and independent random variables $\left(\theta_{n}\right)_{n}$, one
has
\begin{equation}
\lim_{N\rightarrow\infty}R\left(\left[-s,s\right],\left(\theta_{n}\right)_{n},N\right)=2s\label{eq: counting function asymtotically Poissonian}
\end{equation}
for every $s\geq0$ almost surely.\\
\\
Wondering about the optimal bound for the additive energy of the truncations
$A_{N}$ to imply the metric PPC property of $\left(a_{n}\right)_{n}$,
the two following questions were raised in \cite{Aistleitner Larcher Lewko: Additive Energy and the Hausdorff Dimension of the Exceptional Set in Metric Pair Correlation Problems}
where we use the convention that $f=\Omega\left(g\right)$ means for
$f,g:\mathbb{N}\rightarrow\mathbb{R}$ there is a constant $c>0$
such that $g\left(n\right)>cf\left(n\right)$ holds for infinitely
many $n$. 
\begin{namedthm}[Question 1]
Is it possible for a strictly increasing sequence $\left(a_{n}\right)_{n}$
of positive integers with $E\left(A_{N}\right)=\Omega\left(N^{3}\right)$
to have metric PPC?
\begin{namedthm}[Question 2]
Do all increasing strictly sequences $\left(a_{n}\right)_{n}$ of
positive integers with $E\left(A_{N}\right)=o\left(N^{3}\right)$
have metric PPC? 
\end{namedthm}
\end{namedthm}
\noindent Both questions were answered in the negative by Bourgain
whose proofs can be found in \cite{Aistleitner Larcher Lewko: Additive Energy and the Hausdorff Dimension of the Exceptional Set in Metric Pair Correlation Problems}
as an appendix, without giving an estimate on the measure of the set
that was used to answer Question 1, and without a quantitative bound
on $E\left(A_{N}\right)$ appearing in the negation of Question 2.
However, a quantitative analysis, as noted in \cite{Walker: The Primes are not Metric Poissonian},
shows that the sequence Bourgain constructed for Question 2 satisfies
\begin{equation}
E\left(A_{N}\right)=\mathcal{O}_{\varepsilon}\left(\frac{N^{3}}{\left(\log\log N\right)^{\frac{1}{4}+\varepsilon}}\right)\label{eq: Bourgains bound for the sequence of non PPC}
\end{equation}
for any fixed $\varepsilon>0$. Moreover, Nair posed the problem\footnote{This problem was posed at the problem session of the ELAZ conference
in 2016.} whether the sequence of prime numbers $\left(p_{n}\right)_{n}$,
ordered by increasing value, has metric PPC. Recently, Walker \cite{Walker: The Primes are not Metric Poissonian}
answered this question in the negative. Thereby he gave a significantly
better bound than (\ref{eq: Bourgains bound for the sequence of non PPC})
for the additive energy $E\left(A_{n}\right)$ for a sequence $\left(a_{n}\right)_{n}$
not having metric PPC - since the additive energy of the truncations
of $\left(p_{n}\right)_{n}$ is in $\Theta\bigl(\left(\log N\right)^{-1}N^{3}\bigr)$
where $f=\Theta\left(g\right)$, for functions $f,g$, means that
$f=\mathcal{O}\left(g\right)$ and $g=\mathcal{O}\left(f\right)$
holds. The main objective of our work is to improve upon these answers
to Questions 1, and Question 2. 

For a given sequence $\left(a_{n}\right)_{n}$, we denote by $\NPPC\left(\left(a_{n}\right)_{n}\right)$
the (``exceptional'') set of all $\alpha\in\left(0,1\right)$ such
that the pair correlation function (\ref{eq: definition of the Pair Correlation Counting function})
does not tend to $2s$, as $N$ tends to infinity, for some $s\geq0$. 
\begin{namedthm}[Theorem A (Bourgain, \cite{Aistleitner Larcher Lewko: Additive Energy and the Hausdorff Dimension of the Exceptional Set in Metric Pair Correlation Problems})]
\label{thm: Bourgain's Result concerning the measure of the set of counterexamples}Suppose
$\left(a_{n}\right)_{n}$ is a strictly increasing sequence of positive
integers. If $E(A_{N})=\Omega\left(N^{3}\right)$, then $\NPPC\left(\left(a_{n}\right)_{n}\right)$
has positive Lebesgue measure.
\end{namedthm}
We prove the following sharpening.
\begin{thm}
\label{thm: full measure of set of counterexamples}Suppose $\left(a_{n}\right)_{n}$
is a strictly increasing sequence of positive integers. If $E(A_{N})=\Omega\left(N^{3}\right)$,
then $\NPPC\left(\left(a_{n}\right)_{n}\right)$ has full Lebesgue
measure.
\end{thm}
Moreover, we lower the known energy threshold, and estimate the Hausdorff
dimension of the exceptional set from below. For stating our second
main theorem, we denote by $\mathbb{R}_{>x}$ the set of real numbers
exceeding a given $x\in\mathbb{R}$, and recall that for a function
$g:\mathbb{R}_{>1}\rightarrow\mathbb{R}_{>0}$ the lower order of
infinity $\lambda\left(g\right)$ is defined by
\[
\lambda\left(g\right)\coloneqq\liminf_{x\rightarrow\infty}\frac{\log g\left(x\right)}{\log x}.
\]
\begin{rem*}
This notion arises naturally in the context of Hausdorff dimensions.
Roughly speaking, it quantifies the (lower) asymptotic growth rate
of a function.
\end{rem*}
\begin{thm}
\label{thm: lowering the known Energy threshold}Let $f:\mathbb{R}_{>0}\rightarrow\mathbb{R}_{>2}$
be a function increasing monotonically to $\infty$, and satisfying
$f\left(x\right)=\mathcal{O}\bigl(\left(\log x\right)^{-\nicefrac{7}{3}}x^{\nicefrac{1}{3}}\bigr)$.
Then, there exists a strictly increasing sequence $\left(a_{n}\right)_{n}$
of positive integers with $E(A_{N})=\Theta\bigl((f\left(N\right))^{-1}N^{3}\bigr)$
such that if 
\begin{equation}
\sum_{n\geq1}\frac{1}{nf(n)}\label{eq: divergence of the reciprocal of (f(n) times n)}
\end{equation}
diverges, then for Lebesgue almost all $\alpha\in\left[0,1\right]$
\begin{equation}
\limsup_{N\rightarrow\infty}R\left(\left[-s,s\right],\alpha,N\right)=\infty\label{eq: divergence of the Pair Correlation Function}
\end{equation}
holds for any $s>0$; additionally, if (\ref{eq: divergence of the reciprocal of (f(n) times n)})
converges and $\sup\left\{ f\left(2x\right)/f\left(x\right):\,x\geq x_{0}\right\} $
is strictly less than $2$ for some $x_{0}>0$, then $\NPPC\left(\left(a_{n}\right)_{n}\right)$
has Hausdorff dimension at least $\left(1+\lambda\right)^{-1}$ where
$\lambda$ is the lower order of infinity of $f$. 
\end{thm}
We record an immediate consequence of Theorem \ref{thm: lowering the known Energy threshold}
by using the convention that the $r$-folded iterated logarithm is
denoted by $\log_{r}\left(x\right)$, i.e. $\log_{r}\left(x\right)\coloneqq\log_{r-1}\left(\log\left(x\right)\right)$
and $\log_{1}\left(x\right)\coloneqq\log\left(x\right)$.
\begin{cor}
\label{cor: order of magnitude for the additive energy of the sequence of counter examples}Let
$r$ be a positive integer. Then, there is a strictly increasing sequence
$\left(a_{n}\right)_{n}$ of positive integers with 
\[
E\left(A_{N}\right)=\Theta\left(\frac{N^{3}}{\log\left(N\right)\log_{2}\left(N\right)\ldots\log_{r}\left(N\right)}\right)
\]
such that $\NPPC\left(\left(a_{n}\right)_{n}\right)$ has full Lebesgue
measure. Moreover, for any $\varepsilon>0$ there is a strictly increasing
sequence $\left(a_{n}\right)_{n}$ of positive integers with 
\[
E\left(A_{N}\right)=\Theta\left(\frac{\left(\log_{r}\left(N\right)\right)^{-\varepsilon}N^{3}}{\log\left(N\right)\log_{2}\left(N\right)\ldots\log_{r}\left(N\right)}\right)
\]
such that $\NPPC\left(\left(a_{n}\right)_{n}\right)$ has full Hausdorff
dimension.
\end{cor}
The proof of Theorem \ref{thm: lowering the known Energy threshold}
connects the metric PPC property to the notion of ``optimal regular
systems'' from Diophantine approximation. It uses, among other things,
a Khintchine-type theorem due to Beresnevich. Furthermore, despite
leading to better bounds, the nature of the sequences underpinning
Theorem \ref{thm: lowering the known Energy threshold} is much simpler
than the nature of those sequences previously constructed by Bourgain
\cite{Aistleitner Larcher Lewko: Additive Energy and the Hausdorff Dimension of the Exceptional Set in Metric Pair Correlation Problems}
(who used, inter alia, large deviations inequalities form a probability
theory), or the sequence of prime numbers studied by Walker \cite{Walker: The Primes are not Metric Poissonian}
(who relied on estimates, derived by the circle-method, on the exceptional
set in Goldbach-like problems). \\
\\
In the converse direction, there has been remarkable progress, due
to a work of Bloom, Chow, Gafni, Walker - who improved under the assumption
that the sequence is not ``too sparse'' the power saving bound (\ref{eq: Aistleitner bound})
to a saving of a little more than the square of a logarithm. More
precisely, their result is as follows.
\begin{namedthm}[Theorem B (Bloom, Chow, Gafni, Walker \cite{Bloom-Chow-Gafni-Walker: Additive Energy and the Metric Poissonian}]
\label{thm: Bloom, Chow, Gafni, Walker theorem}Let $\left(a_{n}\right)_{n}$
be a strictly increasing sequence of positive integers. Suppose there
is an $\varepsilon>0$ and a $C=C\left(\varepsilon\right)>0$ such
that 
\[
E\left(A_{N}\right)=\mathcal{O}_{\varepsilon}\left(\frac{N^{3}}{\left(\log N\right)^{2+\varepsilon}}\right),\qquad\delta\left(N\right)\geq\frac{C}{\left(\log N\right)^{2+2\varepsilon}}
\]
where $\delta\left(N\right)\coloneqq N^{-1}\#\left(A_{N}\cap\left\{ 1,\ldots,N\right\} \right)$.
Then, $\left(a_{n}\right)_{n}$ has metric PPC.
\end{namedthm}

\section{First main theorem}

Let us give an outline of the proof of Theorem \ref{thm: full measure of set of counterexamples}.
For doing so, we begin by sketching the reasoning of Theorem A: As
it turns out, except for a set of neglectable measure, the counting
function in (\ref{eq: definition of the Pair Correlation Counting function})
can be written as a function that admits a non-trivial estimate for
its $L^{1}$-mean value. The $L^{1}$-mean value is infinitely often
too small on sets whose measure is uniformly bounded from below. Thus,
there exists a sequence of set $\left(\Omega_{r}\right)_{r}$ of $\alpha\in\left[0,1\right]$
such that $R\left(\left[-s,s\right],\alpha,N\right)$ is too small
for every $\alpha\in\Omega_{r}$ for having PPC and Theorem A follows.

Our reasoning for proving Theorem \ref{thm: full measure of set of counterexamples}
is building upon this argument of Bourgain while we introduce new
ideas to construct a sequence of sets $\left(\Omega_{r}\right)_{r}$
that are ``quasi (asymptotically) independent'' - meaning that for
every fixed $t$ the relation $\lambda(\Omega_{r}\cap\Omega_{t})\leq\lambda(\Omega_{r})\lambda(\Omega_{t})+o\left(1\right)$
holds as $r\rightarrow\infty$. Roughly speaking, applying a suitable
version of the Borel-Cantelli lemma, combined with a sufficiently
careful treatment of the $o\left(1\right)$ term, will then yield
Theorem \ref{thm: full measure of set of counterexamples}. However,
before proceeding with the details of the proof we collect in the
next paragraph some tools from additive combinatorics that are needed. 

\subsection{Preliminaries}

We start with a well-know result relating, in a quantitative manner,
the additive energy of a set of integers with the existence of a (relatively)
dense subset with small difference set where the difference set $B-B\coloneqq\left\{ b-b':\,b,b'\in B\right\} $
for a set $B\subseteq\mathbb{R}$. 
\begin{lem}[{Balog-Szeméredi-Gowers lemma, \cite[Thm 2.29]{Tao Vu: Additive combinatorics}}]
\label{Balog-Szem=0000E9redi-Gowers}Let $A\subseteq\mathbb{Z}$
be a finite set of integers. For any $c>0$ there exist $c_{1},c_{2}>0$
depending only on $c$ such that the following holds. If $E(A)\geq c\left(\#A\right)^{3}$,
then there is a subset $B\subseteq A$ such that 
\begin{enumerate}
\item $\#B\geq c_{1}\#A,$
\item $\#\left(B-B\right)\leq c_{2}\#A.$
\end{enumerate}
\end{lem}
Moreover, we recall that for $\delta>0$ and $d\in\mathbb{Z}$ the
set 
\[
B\left(d,\delta\right)\coloneqq\left\{ \alpha\in\left[0,1\right]:\,\left\Vert d\alpha\right\Vert \leq\delta\right\} 
\]
is called Bohr set. These appear frequently in additive combinatorics.
The following two simple observation will be useful.
\begin{lem}
\label{lem: upper estimate for measure of Omega_varepsilon,n}Let
$B\subseteq\mathbb{Z}$ be a finite set of integers. Then,

\[
\lambda\Biggl(\Biggl\{\alpha\in\left[0,1\right]:\underset{d\in\left(B-B\right)\setminus\left\{ 0\right\} }{\min}\left\Vert d\alpha\right\Vert <\frac{\varepsilon}{\#\left(B-B\right)}\Biggr\}\Biggr)\leq2\varepsilon
\]
for every $\varepsilon\in(0,1)$ where $\lambda$ is the Lebesgue
measure. 
\end{lem}
\begin{proof}
By observing that the set under consideration is contained in
\[
\bigcup_{\underset{m\not=n}{m,n\in B}}B\left(m-n,\frac{\varepsilon}{\#\left(B-B\right)}\right),
\]
and $\lambda\left(B\left(m-n,\frac{\varepsilon}{\#\left(B-B\right)}\right)\right)=\frac{2\varepsilon}{\#\left(B-B\right)}$,
the claim follows at once.
\end{proof}
\begin{lem}
\label{lem: Omega_n has only finitely many connected components}Suppose
$A$ is a finite intersection of Bohr sets, and $B$ is a finite union
of Bohr sets. Then, $A\setminus B$ is the union of finitely many
intervals.
\end{lem}
Furthermore, we shall use the Borel-Cantelli lemma in a version due
to Erd\H{o}s-Rényi.
\begin{lem}[Erd\H{o}s-Rényi]
\label{lem: Erdos Renyi version of Borel Cantelli}Let $\left(A_{n}\right)_{n}$
be a sequence of Lebesgue measurable sets in $\left[0,1\right]$ satisfying
\[
\sum_{n\geq1}\lambda\left(A_{n}\right)=\infty.
\]
Then,

\[
\lambda\left(\limsup_{n\rightarrow\infty}A_{n}\right)\geq\limsup_{N\rightarrow\infty}\frac{\left(\sum_{n\leq N}\lambda\left(A_{n}\right)\right)^{2}}{\sum_{m,n\leq N}\lambda\left(A_{n}\cap A_{m}\right)}.
\]
\end{lem}
Moreover, let us explain the main steps in the proof of Theorem \ref{thm: full measure of set of counterexamples}.
Let $\varepsilon\coloneqq\varepsilon\left(j\right)\coloneqq\frac{1}{10^{j}}c_{1}^{2}$
be for $j\in\mathbb{N}$ where the constant $c_{1}$ is specified
later-on, and fix $j$ for now. In the first part of the argument,
we show how a sequence - that is constructed in the second part of
the argument - with the following crucial (but technical) properties
implies the claim. For every fixed $j$, we find a corresponding $s=s(j)$
and construct a sequence $\left(\Omega_{r}\right)_{r}$ of exceptional
values $\alpha$ satisfying the following properties:
\begin{enumerate}[label={(}\roman{enumi}{)}]
\item \label{enu:Pair correlations functions too small on exceptional set}\label{enu:First property of exceptional sets}For
all $\alpha\in\Omega_{r}$, the pair correlation function admits the
upper bound
\begin{equation}
R\left(\left[-s,s\right],\alpha,N\right)\leq2\tilde{c}s\label{eq: Pair correlations function too small on exceptional set}
\end{equation}
for some absolute constant $\tilde{c}\in\left(0,1\right)$, depending
on $\left(a_{n}\right)$ only. 
\item \label{enu:exceptional sets get upper asymptotically independent}For
all integers $r>t\ge1$, the relation
\begin{equation}
\lambda\left(\Omega_{r}\cap\Omega_{t}\right)\leq\lambda\left(\Omega_{r}\right)\lambda\left(\Omega_{t}\right)+2\varepsilon\lambda\left(\Omega_{t}\right)+\mathcal{O}\left(r^{-2}\right)\label{eq: exceptional sets get upper asymptotically independent}
\end{equation}
holds. 
\item \label{enu:Each exceptional set has only finitely many connected components}Each
$\Omega_{r}$ is the union of finitely many intervals (hence measurable).
\item \label{enu: absolute lower bound for the measure of Omega}\label{enu: last property of exceptional sets}For
all $r\geq1$, the measure $\lambda\left(\Omega_{r}\right)$ is uniformly
bounded from below by
\begin{equation}
\lambda\left(\Omega_{r}\right)\geq\frac{c_{1}^{2}}{8}.\label{eq: absolute lower bound for the measure of Omega}
\end{equation}
\end{enumerate}

\subsection{Proof of Theorem \ref{thm: full measure of set of counterexamples}}

1. Suppose there is $\left(\Omega_{r}\right)_{r}$ satisfying \ref{enu:Pair correlations functions too small on exceptional set}-\ref{enu: absolute lower bound for the measure of Omega}.
Then, by using (\ref{eq: exceptional sets get upper asymptotically independent}),
we get
\begin{align*}
\sum_{r,t\leq N}\lambda\left(\Omega_{r}\cap\Omega_{t}\right) & \le2\sum_{2\leq t\leq N}\,\sum_{1\leq r<t}\left(\lambda\left(\Omega_{r}\right)\lambda\left(\Omega_{t}\right)\right)+2\varepsilon N^{2}+\mathcal{O}\left(N\right)\\
 & \leq\left(\sum_{t\leq N}\lambda\left(\Omega_{t}\right)\right)^{2}+2\varepsilon N^{2}+\mathcal{O}\left(N\right).
\end{align*}
By recalling that $\Omega_{r}=\Omega_{r}\left(\varepsilon\right)=\Omega_{r}\left(j\right)$,
we let $\Omega(j)\coloneqq\limsup_{r\rightarrow\infty}\Omega_{r}$.
By using the inequality above in combination with Lemma \ref{lem: Erdos Renyi version of Borel Cantelli}
and the bound (\ref{eq: absolute lower bound for the measure of Omega}),
we obtain that the set $\Omega(j)$ has measure at least
\begin{align*}
\limsup_{N\rightarrow\infty}\frac{\left(\sum_{r\leq N}\lambda\left(\Omega_{r}\right)\right)^{2}}{\sum_{r,t\leq N}\lambda\left(\Omega_{r}\cap\Omega_{t}\right)} & \geq\limsup_{N\rightarrow\infty}\frac{1}{1+\frac{4\varepsilon N^{2}}{\left(\sum_{r\leq N}\lambda\left(\Omega_{r}\right)\right)^{2}}}\\
 & \geq\limsup_{N\rightarrow\infty}\frac{1}{1+\frac{256}{c_{1}^{2}}\varepsilon}=\frac{1}{1+\frac{256}{c_{1}^{2}}\varepsilon}.
\end{align*}
Note that due to (\ref{eq: Pair correlations function too small on exceptional set})
every $\alpha\in\Omega\left(j\right)$ does not have PCC. Now, letting
$j\rightarrow\infty$ proves the assertion.\\

\noindent2. For constructing $\left(\Omega_{r}\right)_{r}$ with
the required properties, let $c>0$ such that $E\left(A_{N}\right)>cN^{3}$
for infinitely many integers $N$. By choosing an appropriate subsequence
$\left(N_{i}\right)_{i}$ and omitting the subscript $i$ for ease
of notation, $E\left(A_{N}\right)>cN^{3}$ holds for every $N$ occurring
in this proof. Moreover, let $c_{1},c_{2}$ and $B_{N}$ be as in
Lemma \ref{Balog-Szem=0000E9redi-Gowers}, corresponding to the $c$
just mentioned. Arguing inductively, while postponing the base step,\footnote{The bases step uses simplified versions of the arguments exploited
in the induction step, and will therefore be postponed.} we assume that for $1\leq r<R$, and $s=\frac{\varepsilon}{2c_{2}}$
there are sets $\left(\Omega_{r}\right)_{1\leq r<R}$ that satisfy
the properties \ref{enu:First property of exceptional sets}-\ref{enu: last property of exceptional sets}
for all distinct integers $1\leq r,t<R$. Let $N\geq R$. Lemma \ref{lem: upper estimate for measure of Omega_varepsilon,n}
implies that the set $\Omega_{\varepsilon,N}\subseteq[0,1]$ of all
$\alpha\in\left[0,1\right]$ satisfying $\left\Vert \left(r-t\right)\alpha\right\Vert <N^{-1}s$
for some distinct $r,t\in B_{N}$ has measure at most $2\varepsilon$.%
\begin{comment}
hier den Induktionsanfang beschreiben!! Das Lemma scheint nur an dieser
Stelle einzugehen. Das könnten wir rausstreichen.
\end{comment}
{} Setting
\[
\mathcal{D}_{N}:=\left\{ \left(r,t\right)\in\left(A_{N}\times A_{N}\right)\setminus\left(B_{N}\times B_{N}\right):\,r\not=t\right\} ,
\]
we get for $\alpha\notin\Omega_{\varepsilon,N}$ that 
\[
R\left(\left[-s,s\right],\alpha,N\right)=\frac{1}{N}\#\left\{ \left(r,t\right)\in\mathfrak{\mathcal{D}}_{N}:\,\left\Vert \left(r-t\right)\alpha\right\Vert <N^{-1}s\right\} .
\]
Let $\ell_{R}$ denote the length of the smallest subinterval of $\Omega_{r}$
for $1\leq r<R$, and define $C\left(\Omega_{r}\right)$ to be the
set of subintervals of $\Omega_{r}$. Note that $\ell_{R}>0$, and
$\max_{1\leq r<R}\#C\left(\Omega_{r}\right)<\infty$. We divide $\left[0,1\right)$
into 
\[
P\coloneqq\left\lfloor 1+2\ell_{R}^{-1}R^{2}\max_{1\leq r<R}\#C\left(\Omega_{r}\right)\right\rfloor 
\]
parts $\mathcal{P}_{i}$ of equal lengths, i.e. $\mathcal{P}_{i}\coloneqq\left[\frac{i}{P},\frac{i+1}{P}\right)$
where $i=0,\ldots,P-1$. After writing 
\begin{align}
 & \frac{1}{N}\underset{\mathcal{P}_{i}}{\int}\#\left\{ \left(r,t\right)\in\mathfrak{\mathcal{D}}_{N}:\,\left\Vert \left(r-t\right)\alpha\right\Vert \leq N^{-1}s\right\} \text{d}\alpha\label{eq: counting function integrated on an atom}\\
 & =\frac{1}{N}\underset{\left(r,t\right)\in\mathfrak{\mathcal{D}}_{N}}{\sum}\underset{\mathcal{P}_{i}}{\int}\mathbf{1}_{\left[-\frac{s}{N},\frac{s}{N}\right]}\left(\left\Vert \left(r-t\right)\alpha\right\Vert \right)\text{d}\alpha,\nonumber 
\end{align}
we split the sum into two parts: one part containing differences $\left|r-t\right|>R^{k}P$,
and a second part containing differences $\left|r-t\right|\leq R^{k}P$
where 
\[
k\coloneqq\left\lfloor \frac{1}{\log2}\log\frac{20}{c_{1}^{2}\left(1-2^{-1}c_{1}^{2}\right)s}\right\rfloor +1.
\]
Letting $\mathbf{1}_{B}$ denote the characteristic function of $X\subseteq\left[0,1\right]$,
the Cauchy-Schwarz inequality implies 
\[
\underset{\mathcal{P}_{i}}{\int}\mathbf{1}_{\left[-\frac{s}{N},\frac{s}{N}\right]}\left(\left\Vert \left(r-t\right)\alpha\right\Vert \right)\text{d}\alpha\leq\sqrt{\frac{1}{P}\frac{2s}{N}}.
\]
Since for any $x>0$ there are at most $2xN$ choices of $\left(r,t\right)\in\mathcal{D}_{N}$
such that $\left|r-t\right|\leq x$, we obtain 
\[
\frac{1}{N}\underset{\underset{\left|r-t\right|\leq PR^{k}}{\left(r,t\right)\in\mathcal{D}_{N}}}{\sum}\underset{\mathcal{P}_{i}}{\int}\mathbf{1}_{\left[-\frac{s}{N},\frac{s}{N}\right]}\left(\left\Vert \left(r-t\right)\alpha\right\Vert \right)\text{d}\alpha\leq2PR^{k}\sqrt{\frac{1}{P}\frac{2s}{N}}
\]
which is $\leq P^{-1}R^{-k}$ if $N$ is sufficiently large. Moreover,
for any $\left|r-t\right|>PR^{k}$ we observe that
\[
\underset{\mathcal{P}_{i}}{\int}\mathbf{1}_{\left[-\frac{s}{N},\frac{s}{N}\right]}\left(\left\Vert \left(r-t\right)\alpha\right\Vert \right)\text{d}\alpha\leq\frac{2s}{PN}+\frac{4}{PR^{k}N}
\]
and $\#\mathcal{D}_{N}\leq N^{2}-\bigl(\#B_{N}\bigr)^{2}\leq\tilde{c}N^{2}$
where $\tilde{c}\coloneqq1-c_{1}^{2}$.\textcolor{red}{{} }Therefore,
the mean value (\ref{eq: counting function integrated on an atom})
on $\mathcal{P}_{i}$ of the counting function $R$ is bounded from
above by
\begin{align*}
\frac{1}{N}\left(\#\mathcal{D}_{N}\right)^{2}\left(\frac{2s}{PN}+\frac{4}{PR^{k}N}\right)+\frac{1}{PR^{k}}\leq\frac{2\tilde{c}s}{P}+\frac{5}{PR^{k}}.
\end{align*}
Hence, it follows that the measure of the set $\Delta_{N}\left(i\right)$
of $\alpha\in\mathcal{P}_{i}$ with
\begin{equation}
\frac{1}{N}\#\left\{ \left(r,t\right)\in\mathfrak{\mathcal{D}}_{N}:\,\left\Vert \left(r-t\right)\alpha\right\Vert \leq N^{-1}s\right\} \leq2\left(1-\frac{c_{1}^{2}}{2}\right)s\label{eq: modified counting too small for being Poissonian}
\end{equation}
admits, by the choice of $k$, the lower bound
\begin{equation}
\lambda\left(\Delta_{N}\left(i\right)\right)\geq\frac{1}{P}-\frac{1}{P}\frac{2\tilde{c}s+5R^{-k}}{2\left(1-\frac{c_{1}^{2}}{2}\right)s}\geq\frac{1}{P}\left(\frac{c_{1}^{2}}{2}-\frac{c_{1}^{2}}{8}\right).\label{eq: absolute lower bound for Omega_N on partition}
\end{equation}
Note that $\Delta_{N}\left(i\right)$ is the union of finitely many
intervals, due to Lemma \ref{lem: Omega_n has only finitely many connected components}.
So, we may take $\Delta_{N}'\left(i\right)\subset\Delta_{N}\left(i\right)$
being a finite union of intervals such that $\lambda\left(\Delta_{N}'\left(i\right)\right)$
equals the lower bound in (\ref{eq: absolute lower bound for Omega_N on partition}).
Let 
\[
\Omega_{R}\coloneqq\Omega_{R}\left(N\right)\coloneqq\Delta_{N}\setminus\Omega_{\varepsilon,N}\qquad\mathrm{where}\qquad\Delta_{N}\coloneqq\bigcup_{i=0}^{P-1}\Delta_{N}'\left(i\right).
\]
We are going to show now that $\Omega_{R}$ satisfies the properties
\ref{enu:First property of exceptional sets} - \ref{enu: last property of exceptional sets}.
Now, $\Omega_{R}$ satisfies property \ref{enu: absolute lower bound for the measure of Omega}
with $r=R$ since
\[
\lambda\left(\Omega_{R}\right)\geq\lambda\left(\Delta_{N}\right)-\lambda\left(\Omega_{\varepsilon,N}\right)=\frac{c_{1}^{2}}{2}-\frac{c_{1}^{2}}{8}-2\varepsilon\geq\frac{c_{1}^{2}}{8}.
\]
Furthermore, $\Omega_{R}$ satisfies property \ref{enu:Pair correlations functions too small on exceptional set}
by construction and also property \ref{enu:Each exceptional set has only finitely many connected components}
since all sets involved in the construction of $\Omega_{R}$ were
a finite union of intervals. Let $1\leq r<R$, and $I$ be a subinterval
of $\Omega_{r}$. Then,
\begin{align*}
\lambda\left(I\cap\Delta_{N}\right) & =\sum_{i:\mathcal{P}_{i}\cap I\neq\emptyset}\lambda\left(\mathcal{P}_{i}\cap I\cap\Delta_{N}\right)\\
 & \leq\frac{2}{P}+\sum_{i:\mathcal{P}_{i}\subsetneq I}\lambda\left(\mathcal{P}_{i}\cap\Delta_{N}\right)\\
 & \leq\frac{2}{P}+\sum_{i:\mathcal{P}_{i}\subsetneq I}\lambda\left(\Delta_{N}'\left(i\right)\right).
\end{align*}
By summing over all subintervals $I\in C\left(\Omega_{r}\right)$,
we obtain that 
\begin{align*}
\lambda\left(\Omega_{r}\cap\Delta_{N}\right) & \leq\sum_{I\in C\left(\Omega_{r}\right)}\left(\frac{2}{P}+\sum_{i:\mathcal{P}_{i}\subsetneq I}\lambda\left(\Delta_{N}'\left(i\right)\right)\right)\\
 & \leq\frac{1}{R^{2}}+\sum_{I\in C\left(\Omega_{r}\right)}P\lambda\left(I\right)\frac{\lambda\left(\Omega_{N}\right)}{P}\\
 & =\lambda\left(\Omega_{r}\right)\lambda\left(\Omega_{N}\right)+\frac{1}{R^{2}}.
\end{align*}
We deduce property \ref{enu:exceptional sets get upper asymptotically independent}
from this estimate and Lemma \ref{lem: upper estimate for measure of Omega_varepsilon,n}
via
\begin{align*}
\lambda\left(\Omega_{r}\cap\Omega_{R}\right) & \leq\lambda\left(\Omega_{r}\cap\Delta_{N}\right)\\
 & \leq\lambda\left(\Omega_{r}\right)\left(\lambda\left(\Omega_{N}\right)-\lambda\left(\Omega_{\varepsilon,N}\right)\right)+\frac{1}{R^{2}}+\lambda\left(\Omega_{r}\right)\lambda\left(\Omega_{\varepsilon,N}\right)\\
 & \leq\lambda\left(\Omega_{r}\right)\lambda\left(\Omega_{R}\right)+2\varepsilon\lambda\left(\Omega_{r}\right)+\mathcal{O}\left(R^{-2}\right)
\end{align*}
This concludes the induction step. The only part missing now is the
base step of the induction. For realizing it, let $N$ denote the
smallest integer $m$ with $E\left(A_{m}\right)>cm^{3}$. We replace
$\mathcal{P}_{i}$ in (\ref{eq: counting function integrated on an atom})
by $\left[0,1\right]$ to directly derive 
\[
\int_{0}^{1}\frac{1}{N}\#\left\{ \left(r,t\right)\in\mathfrak{\mathcal{D}}_{N}:\,\left\Vert \left(r-t\right)\alpha\right\Vert \leq N^{-1}s\right\} \mathrm{d}\alpha\leq2\tilde{c}s,
\]
and conclude that the set $\Omega_{1}'$ of $\alpha\in\left[0,1\right]$
satisfying (\ref{eq: modified counting too small for being Poissonian})
has a measure at least $\frac{c_{1}^{2}}{2}$. Thus, $\Omega_{1}\coloneqq\Omega_{1}'\setminus\Omega_{N,\varepsilon}$
has measure at least as large as the right hand side of (\ref{eq: absolute lower bound for the measure of Omega}).
For property (\ref{eq: exceptional sets get upper asymptotically independent})
is nothing to check and that $\Omega_{1}$ is a finite union of intervals
follows from Lemma \ref{lem: Omega_n has only finitely many connected components}
by observing that 
\[
\Omega_{1}'=\bigcap_{d_{1},\ldots,d_{L\left(N\right)}}\left(B\left(d_{1},N^{-1}s\right)^{C}\cup\ldots\cup B\left(d_{L\left(N\right)},N^{-1}s\right)^{C}\right)
\]
where the intersection runs through any set of $L\left(N\right)=\left\lfloor N2\tilde{c}s\right\rfloor $
tuples of differences $d_{i}=r_{i}-t_{i}\neq0$ of components of $\left(r_{i},t_{i}\right)\in\mathcal{D}_{N}$
for $i=1,\ldots,L\left(N\right)$. 

Thus, the proof is complete.

\section{Second main theorem}

The sequences $\left(a_{n}\right)_{n}$ enunciated in Theorem \ref{thm: lowering the known Energy threshold}
are constructed in two steps. In the first step, we concatenate (finite)
blocks, with suitable lengths, of arithmetic progressions to form
a set $P_{A}$. In the second step, we concatenate (finite) blocks,
with suitable lengths, of geometric progressions to form a set $P_{G}$
and then define $a_{n}$ to be the $n$-th element of $P_{A}\cup P_{G}$.
On the one hand, the arithmetic progression like part $P_{A}$ serves
to ensure, due to considerations from metric Diophantine approximation,
the divergence property (\ref{eq: divergence of the Pair Correlation Function})
on a set with full measure or controllable Hausdorff dimension; on
the other hand, the geometric progression like part $P_{G}$ lowers
the additive energy, as much as it can. For doing so, a geometric
block will appear exactly before and after an arithmetic block, and
have much more elements. \\
\\
For writing the construction precisely down, we introduce some notation.
Let henceforth $\left\lfloor x\right\rfloor $ denote the greatest
integer $m$ that is at most $x\in\mathbb{R}$. Suppose trough-out
this section that $f$ is as in Theorem \ref{thm: lowering the known Energy threshold}.
We set $P_{A}^{\left(1\right)}$ to be the empty set while $P_{G}^{\left(1\right)}\coloneqq\left\{ 1,2\right\} $.
Moreover, for $j\geq2$ we let $P_{A}^{\left(j\right)}$ denote the
set of $\bigl\lfloor2^{j}\bigl(f(2^{j})\bigr)^{-\beta}\bigr\rfloor$
consecutive integers that start with $C_{j}=2\max\bigl\{ P_{G}^{(j-1)}\bigr\}$,
and $P_{G}^{\left(j\right)}$ is such that the difference set $P_{G}^{\left(j\right)}-2C_{j}$
is the geometric progression $2^{i}$ for $1\leq i\leq\bigl\lfloor\bigl(f(2^{j})\bigr)^{-\gamma}2^{j}\bigl(1-\bigl(f(2^{j})\bigr)^{\gamma-\beta}\bigr)\bigr\rfloor$
where $0<\gamma<\beta<\nicefrac{3}{4}$ are parameters\footnote{No particular importance should be attached to requiring $\beta<\nicefrac{3}{4}$,
or using ``dyadic steps lengths $2^{j}$''. Doing so is for simplifying
the technical details only - eventually, it will turn out that $\beta=\nicefrac{2}{3}=2\gamma$
is the optimal choice of parameters in this approach. For proving
this to the reader, we leave $\gamma,\beta$ undetermined till the
end of this section.} to be chosen later-on. In this notation, we take
\[
P_{A}\coloneqq\bigcup_{j\geq1}P_{A}^{\left(j\right)},\qquad P_{G}\coloneqq\bigcup_{j\geq1}P_{G}^{\left(j\right)},
\]
and denote by $a_{n}$ the $n$-th smallest element in $P_{A}\cup P_{G}$.
 For $d\in\mathbb{Z}$ and finite sets of integers $X,Y$, we abbreviate
the number of representation of $d$ as a difference of an $x\in X$
and a $y\in Y$ by $\text{rep}_{X,Y}(d)\coloneqq\#\{(x,y)\in X\times Y:\,x-y=d\}$;
observe that 
\begin{equation}
E\left(X\right)=\sum_{d\in\mathbb{Z}}\left(\mathrm{rep}_{X,X}\left(d\right)\right)^{2},\label{eq: additive Energy in terms of number of representation}
\end{equation}
and 
\begin{equation}
R\left(\left[-s,s\right],\alpha,N\right)=\frac{1}{N}\underset{d\neq0}{\sum}\text{rep}_{A_{N},A_{N}}(d)\mathbf{1}_{\left[0,\frac{s}{N}\right]}\left(\left\Vert \alpha d\right\Vert \right).\label{eq: lower bound for counting function fo the pair correlations}
\end{equation}

\subsection{Preliminaries}

We begin to determine the order of magnitude of $E\left(A_{N}\right)$
for the truncations $A_{N}$ of the sequence constructed above. Since
the cardinality of elements in the union of the blocks $P_{G}^{\left(j\right)},P_{A}^{\left(j\right)}$
has about exponential growth, it is reasonable to expect $E\left(A_{N}\right)$
to be of the same order of magnitude as the additive energy of the
last block $P_{G}^{\left(J\right)}\cup P_{A}^{\left(J\right)}$ that
is fully contained in $A_{N}$ - note that $J=J\left(N\right)$; i.e.
to expect the magnitude of $E\bigl(P_{G}^{\left(J\right)}\cup P_{A}^{\left(J\right)}\bigr)$
which is roughly equal to $E\bigl(P_{A}^{\left(J\right)}\bigr)$.
The following proposition verifies this heuristic considerations.
\begin{prop}
\label{prop: additive energy of good-guy-bad-guy sequence}Let $\left(a_{n}\right)_{n}$
be as in the beginning of Section 3, and $f$ be as in one of the
two assertions in Theorem \ref{thm: lowering the known Energy threshold}.
Then, $E\left(A_{N}\right)=\Theta\bigl(N^{3}\bigl(f\bigl(N\bigr)\bigr){}^{-3\left(\beta-\gamma\right)}\bigr)$.
\end{prop}
\noindent For the proof of Proposition \ref{prop: additive energy of good-guy-bad-guy sequence},
we need the next technical lemma.
\begin{lem}
\label{lem: auxiliary lemma for calculating additive energy}Let $\FJ\coloneqq2^{j}\bigl(f\bigl(2^{j}\bigr)\bigr)^{-\delta}$,
for $j\geq1$ and fixed $\delta\in\left(0,1\right)$, where $f$ is
as in Proposition \ref{prop: additive energy of good-guy-bad-guy sequence}.
Then, $\sum_{i\leq j}F_{i}=\mathcal{O}\bigl(F_{j}\bigr)$ and 
\[
\sum_{d\in\mathbb{Z}}\biggl(\sum_{j,i\leq J}\mathrm{rep}_{P_{G}^{\left(j\right)},P_{A}^{\left(i\right)}}\left(d\right)\biggr)^{2}=\mathcal{O}\left(J^{6}2^{2J}\right).
\]
\end{lem}
\begin{proof}
Suppose that $f\left(x\right)=\mathcal{O}\bigl(x^{\nicefrac{1}{3}}\left(\log x\right)^{-\nicefrac{7}{3}}\bigr)$
is such that (\ref{eq: divergence of the reciprocal of (f(n) times n)})
diverges. Because
\[
\sum_{j\leq J+1}\frac{1}{f\bigl(2^{j}\bigr)}\geq\sum_{k\leq2^{J}}\frac{1}{kf\left(k\right)}
\]
diverges as $J\rightarrow\infty$ and $\left(f\bigl(2^{j}\bigr)/f\bigl(2^{j+1}\bigr)\right)_{j}$
is non-decreasing, we conclude that $\lim_{j\rightarrow\infty}\bigl(f\bigl(2^{j}\bigr)/f\bigl(2^{j+1}\bigr)\bigr)=1$.
Therefore, there is an $i_{0}$ such that the estimate $\bigl(f\bigl(2^{i}\bigr)\bigr)^{-1}f\bigl(2^{i+h}\bigr)<\bigl(\nicefrac{3}{2}\bigr)^{\frac{h}{\delta}}$
holds for any $i\geq i_{0}$ and $h\in\mathbb{N}$. Hence, 
\[
\frac{1}{F_{j}}\sum_{i\leq j}F_{i}\leq o\left(1\right)+\sum_{i_{0}\leq i\leq j}2^{i-j}\left(\frac{3}{2}\right)^{j-i}=\mathcal{O}\bigl(1\bigr).
\]
If $f$ is such that (\ref{eq: divergence of the reciprocal of (f(n) times n)})
converges and $f\left(2x\right)\leq\left(2-\varepsilon\right)f\left(x\right)$
for $x$ large enough, then we obtain by a similar argument that $\sum_{i\leq j}F_{i}$
is in $\mathcal{O}\bigl(F_{j}\bigr)$. Furthermore, $\mathrm{rep}_{P_{G}^{\left(j\right)},P_{A}^{\left(i\right)}}\left(d\right)=\mathcal{O}\left(i\right)$,
for every $j\geq1$, and non-vanishing for $\mathcal{O}\bigl(2^{2j}\bigr)$
values of $d$ which implies the last claim.
\end{proof}
We can now prove the proposition.
\begin{proof}[Proof of Proposition \ref{prop: additive energy of good-guy-bad-guy sequence}]
Let $\FJ=2^{j}\bigl(f\bigl(2^{j}\bigr)\bigr)^{-\beta}$, $N\geq1$
be large and denote by $J=J\left(N\right)\geq0$ the greatest integer
$j$ such that $P_{G}^{\left(j-1\right)}\subseteq A_{N}$. Since 
\[
E\bigl(A_{N}\bigr)\geq E\bigl(P_{A}^{\left(J-1\right)}\bigr)=\Omega\bigl(N^{3}\bigl(f\bigl(N\bigr)\bigr){}^{-3\left(\beta-\gamma\right)}\bigr),
\]
it remains to show that $E\bigl(A_{N}\bigr)=\mathcal{O}\bigl(N^{3}\bigl(f\bigl(N\bigr)\bigr){}^{-3\left(\beta-\gamma\right)}\bigr)$.
By exploiting (\ref{eq: additive Energy in terms of number of representation}),
\[
E\bigl(A_{N}\bigr)\leq\sum_{d\in\mathbb{Z}}\bigl(\mathrm{rep}_{A_{T_{J}},A_{T_{J}}}\left(d\right)\bigr)^{2}\quad\text{where}\quad T_{J}\coloneqq\#\bigcup_{j\leq J}\left(P_{A}^{\left(j\right)}\cup P_{G}^{\left(j\right)}\right).
\]
Moreover, $\mathrm{rep}_{A_{T_{J}},A_{T_{J}}}\left(d\right)=S_{1}\left(d\right)+S_{2}\left(d\right)$
where $S_{1}\left(d\right)$ abbreviates the mixed sum $\sum_{i,j\leq J}\bigl(\mathrm{rep}_{P_{A}^{\left(j\right)},P_{G}^{\left(i\right)}}\left(d\right)+\mathrm{rep}_{P_{G}^{\left(i\right)},P_{A}^{\left(j\right)}}\left(d\right)\bigr)$
and $S_{2}\left(d\right)$ abbreviates the sum $\sum_{i,j\leq J}\bigl(\mathrm{rep}_{P_{G}^{\left(i\right)},P_{G}^{\left(j\right)}}\left(d\right)+\mathrm{rep}_{P_{A}^{\left(i\right)},P_{A}^{\left(j\right)}}\left(d\right)\bigr)$.
Using that for any real numbers $a,b$ the inequality $\left(a+b\right)^{2}\leq2\bigl(a^{2}+b^{2}\bigr)$
holds, we obtain 
\[
E\bigl(A_{N}\bigr)=\mathcal{O}\biggl(\sum_{d\in\mathbb{Z}}\bigl(S_{1}\left(d\right)\bigr)^{2}+\sum_{d\in\mathbb{Z}}\bigl(S_{2}\left(d\right)\bigr)^{2}\biggr).
\]
Lemma \ref{lem: auxiliary lemma for calculating additive energy}
implies that $\sum_{d\in\mathbb{Z}}\bigl(S_{2}\left(d\right)\bigr)^{2}=\mathcal{O}\bigl(\left(\log N\right)^{6}N^{2}\bigr)$
due to $J=\mathcal{O}\left(\log N\right)$. Moreover, we note that
$\mathrm{rep}_{P_{A}^{\left(i\right)},P_{A}^{\left(j\right)}}\left(d\right)$
is non-vanishing for at most $4F_{J}$ values of $d$ as $i,j\leq J$.
Since $\mathrm{rep}_{P_{A}^{\left(i\right)},P_{A}^{\left(j\right)}}\left(d\right)\leq F_{\min\left(i,j\right)}$
holds, we deduce that
\[
\sum_{i,j\leq J}\mathrm{rep}_{P_{A}^{\left(i\right)},P_{A}^{\left(j\right)}}\left(d\right)=\mathcal{O}\biggl(\sum_{j\leq J}\sum_{i\leq j}F_{i}\biggr).
\]
By Lemma \ref{lem: auxiliary lemma for calculating additive energy},
the right hand side is in $\mathcal{O}\bigl(F_{J}\bigr)$. Since $\mathrm{rep}_{P_{G}^{\left(i\right)},P_{G}^{\left(j\right)}}\left(d\right)\leq1$,
where $i,j\leq J$, is non-vanishing for at most $\mathcal{O}\bigl(T_{J}^{2}\bigr)=\mathcal{O}\left(N^{2}\right)$
values of $d$, we obtain that 
\[
\sum_{d\in\mathbb{Z}}\bigl(S_{1}\left(d\right)\bigr)^{2}=\mathcal{O}\bigl(F_{J}^{3}+\left(\log N\right)^{6}N^{2}\bigr)
\]
which is in $\mathcal{O}\bigl(N^{3}\bigl(f\bigl(N\bigr)\bigr){}^{-3\left(\beta-\gamma\right)}\bigr)$.
Hence, $E\bigl(A_{N}\bigr)=\mathcal{O}\bigl(N^{3}\bigl(f\bigl(N\bigr)\bigr){}^{-3\left(\beta-\gamma\right)}\bigr)$.
\end{proof}
For estimating the measure or the Hausdorff dimension of $\NPPC\left(\left(a_{n}\right)_{n}\right)$
from below, we recall the notion of an optimal regular system. This
notion, roughly speaking, describes sequences of real numbers that
are exceptionally well distributed in any subinterval, in a uniform
sense, of a fixed interval.
\begin{defn*}
Let $J$ be a bounded real interval, and $S=\left(\alpha_{i}\right)_{i}$
a sequence of distinct real numbers. $S$ is called an optimal regular
system in $J$ if there exist constants $c_{1},\,c_{2},\,c_{3}>0$
- depending on $S$ and $J$ only - such that for any $I\subseteq J$
there is an index $Q_{0}=Q_{0}\left(S,I\right)$ such that for any
$Q\geq Q_{0}$ there are indices
\begin{equation}
c_{1}Q\leq i_{1}<i_{2}<\ldots<i_{t}\leq Q\label{eq: property of c1 in optimal regular system definition}
\end{equation}
satisfying $\alpha_{i_{h}}\in I$ for $h=1,\ldots,t$, and
\begin{equation}
\left|\alpha_{i_{h}}-\alpha_{i_{\ell}}\right|\geq\frac{c_{2}}{Q}\label{eq: property of c2 in optimal regular system definition}
\end{equation}
for $1\leq h\neq\ell\leq t$, and 
\begin{equation}
c_{3}\lambda\left(I\right)Q\leq t\leq\lambda\left(I\right)Q.\label{eq: property of c3 in optimal regular system definition}
\end{equation}
\end{defn*}
Moreover, we need the following result(s) due to Beresnevich which
may be thought of as a far reaching generalization of Khintchine's
theorem, and Jarník-Besicovitch theorem in Diophantine approximation.
\begin{thm}[{\cite[Thm. 6.1, Thm. 6.2]{Bugeaud: Approximation by algebraic numbers}}]
\label{thm: Khintchine a la Victor}Suppose $\psi:\mathbb{\mathbb{R}}_{>0}\rightarrow\mathbb{R}_{>0}$
is a continuous, non-increasing function, and $S=\bigl(\alpha_{i}\bigr)_{i}$
an optimal regular system in $\left(0,1\right)$. Let $\mathcal{K}_{S}\left(\psi\right)$
denote the set of $\xi$ in $\left(0,1\right)$ such that $\left|\xi-\alpha_{i}\right|<\psi\left(i\right)$
holds for infinitely many $i$. If
\begin{equation}
\sum_{n\geq1}\psi\left(n\right)\label{eq: sum over psi values}
\end{equation}
diverges, then $\mathcal{K}_{S}\left(\psi\right)$ has full measure.\\
Conversely, if (\ref{eq: sum over psi values}) converges, then $\mathcal{K}_{S}\left(\psi\right)$
has measure zero and the Hausdorff dimension equals the reciprocal
of the lower order of $\frac{1}{\psi}$ at infinity. 
\end{thm}
For a rational $\alpha=\frac{p}{q}$, where $p,q\in\mathbb{Z}$,
$q\neq0$, we denote by $H\left(\alpha\right)$ its (naive) height,
i.e. $H\left(\alpha\right)\coloneqq\max\left\{ \left|p\right|,\left|q\right|\right\} $.
It is well-known that the set of rational numbers in $\left(0,1\right)$,
ordered in classes by increasing height and in each class ordered
by numerically values, gives rise to an optimal regular system in
$\left(0,1\right)$. The following lemma says, roughly speaking, that
this assertion remains true for the set of rationals in $\left(0,1\right)$
whose denominators are members of a special sequence that is not too
sparse in the natural numbers. The proof can be given by modifying
the proof of the classical case, compare \cite[Prop. 5.3]{Bugeaud: Approximation by algebraic numbers};
however, we shall give the details for making this article more self-contained. 
\begin{lem}
\label{lem: optimal regular system}Let $\vartheta:\mathbb{R}_{>0}\rightarrow\mathbb{R}_{>1}$
be monotonically increasing to infinity with $\vartheta\left(x\right)=\mathcal{O}\bigl(x^{\nicefrac{1}{4}}\bigr)$
and $\vartheta\left(2^{j+1}\right)/\vartheta\left(2^{j}\right)\rightarrow1$
as $j\rightarrow\infty$. For each $j\in\mathbb{N}$, we let 
\[
B_{j}\coloneqq\frac{2^{j}}{f\left(2^{j}\right)\sqrt{\vartheta\left(2^{j}\right)}},\qquad b_{j}\coloneqq\frac{2}{3}B_{j}.
\]
Let $S=\bigl(\alpha_{i}\bigr)_{i}$ denote a sequence running through
all rationals in $\left(0,1\right)$ whose denominators are in $M\coloneqq\bigcup_{j\geq1}\bigl\{ n\in\mathbb{N}:\,b_{j}\leq n\leq B_{j}\bigr\}$
such that $i\mapsto H\bigl(\alpha_{i}\bigr)$ is non-decreasing. Then,
$S$ is an optimal regular system in $\left(0,1\right)$.
\end{lem}
\begin{proof}
Let $X\geq2$. There are strictly less than $2X^{2}$ rational numbers
in $\left(0,1\right)$ with height bounded by $X$. We take $J=J\left(X\right)$
to be the largest integer $j\geq1$ such that $B_{j}\leq X$. Then,
for $X$ large enough, there are at least
\begin{align*}
\sum_{j\leq J}\sum_{b_{j}\leq q\leq B_{j}}\varphi\left(q\right) & \geq\sum_{j\leq J}\left(\frac{1}{3\pi^{2}}B_{j}^{2}+\mathcal{O}\left(B_{j}\log B_{j}\right)\right)\\
 & \geq\frac{1}{6\pi^{2}}\frac{2^{2J}}{f^{2}\left(2^{J}\right)\vartheta\left(2^{J}\right)}+\mathcal{O}\left(J2^{J}\right)\\
 & >\left(\frac{X}{5\pi}\right)^{2}
\end{align*}
distinct such rationals in $\left(0,1\right)$ with height not exceeding
$X$. Hence, we obtain $\frac{\sqrt{i}}{2}\leq H\left(\alpha_{i}\right)\leq\sqrt{25\pi^{2}\left(i+1\right)}+1$
for $i$ sufficiently large. Let $Q\in\mathbb{N}$, $I\subseteq\left[0,1\right]$
be a non-empty interval, and let $F$ denote the set of $\xi\in I$
satisfying the inequality $\left\Vert q\xi\right\Vert <Q^{-1}$ with
some $1\leq q\leq\frac{1}{1000}Q$. Note that $F$ has measure at
most 
\[
\sum_{q\leq\frac{1}{1000}Q}\left(\frac{2}{qQ}q\lambda\left(I\right)+\frac{2}{qQ}\right)=\frac{1}{500}\lambda\left(I\right)+\mathcal{O}\left(\frac{\log Q}{Q}\right)<\frac{1}{400}\lambda\left(I\right)
\]
for $Q\geq Q_{0}$ where $Q_{0}=Q_{0}\left(S,I\right)$ is sufficiently
large. Let $\bigl\{\nicefrac{p_{j}}{q_{j}}\bigr\}_{1\leq j\leq t}$
be the set of all rationals $\nicefrac{p_{j}}{q_{j}}\in\left(0,1\right)$
with $q_{j}\in M$, $\frac{1}{1000}Q<q_{j}<Q$ that satisfy
\[
\left|\frac{p_{j}}{q_{j}}-\frac{p_{j'}}{q_{j'}}\right|>\frac{2000}{Q^{2}}
\]
whenever $1\leq j\neq j'\leq t$. Observe that for $J$ as above with
$X=Q$ sufficiently large, it follows that
\[
\left\{ q\in M:\,b_{J}\leq q\leq B_{J}\right\} \subseteq\left\{ \frac{Q}{1000},\frac{Q}{1000}+1,\ldots,Q\right\} 
\]
holds and there are hence at least $\frac{1}{3\pi^{2}}B_{J}^{2}+\mathcal{O}\left(B_{J}\log B_{J}\right)>\frac{1}{400}Q^{2}$
choices of $\nicefrac{p_{j}}{q_{j}}\in\left(0,1\right)$ with $q_{j}\in M$
and $\frac{1}{1000}Q<q_{j}<Q$. Due to $\lambda\left(I\setminus F\right)>\frac{399}{400}\lambda\left(I\right)$,
we conclude $t\geq400\frac{Q^{2}}{4000}\frac{399}{400}\lambda\left(I\right)$.
Thus, taking $c_{1}\coloneqq\nicefrac{1}{1000}$, $c_{2}\coloneqq2000$,
and $c_{3}\coloneqq\frac{399}{4000}$ in (\ref{eq: property of c1 in optimal regular system definition}),
(\ref{eq: property of c2 in optimal regular system definition}) and
(\ref{eq: property of c3 in optimal regular system definition}),
respectively, $S$ is shown to be an optimal regular system.
\end{proof}
Now we can proceed to the proof of Theorem \ref{thm: lowering the known Energy threshold}.

\subsection{Proof of Theorem \ref{thm: lowering the known Energy threshold}}

We argue in two steps depending on whether or not the series (\ref{eq: divergence of the reciprocal of (f(n) times n)})
converges. Proposition (\ref{prop: additive energy of good-guy-bad-guy sequence})
implies the announced $\Theta$-bounds on the additive energy of $A_{N}$,
in both cases.\\
\\
\noindent (i) Suppose (\ref{eq: divergence of the reciprocal of (f(n) times n)})
diverges, and fix $s>0$. Let $\vartheta:\mathbb{R}_{>0}\rightarrow\mathbb{R}_{>1}$
be monotonically increasing to infinity with $\vartheta\left(x\right)=\mathcal{O}\left(x^{\nicefrac{1}{4}}\right)$
such that
\begin{equation}
\psi\left(n\right)\coloneqq\frac{1}{nf\left(n\right)\vartheta\left(n\right)}\label{eq: specification of the appxomation function-1}
\end{equation}
satisfies the divergence condition (\ref{eq: sum over psi values}).
Thus, $\vartheta\left(2^{j}\right)/\vartheta\left(2^{j-1}\right)\rightarrow1$
as $j\rightarrow\infty$. Hence, $S=\left(\alpha_{i}\right)_{i}$
from Lemma \ref{lem: optimal regular system} is an optimal regular
system. Furthermore, if $\alpha_{i}=\frac{m}{n}$, then $i\geq cn^{2}$
holds true with a constant $c=c\left(f,\vartheta\right)>0$ due to
$b_{J}\leq n\leq B_{J}$, for some integer $J$, and 
\[
\sum_{j\leq J-1}\sum_{b_{j}\leq m\leq B_{j}}\varphi\left(m\right)=\Theta\bigl(B_{J}^{2}\bigr).
\]
Therefore, $\psi\left(i\right)\leq c^{-1}n^{-2}\bigl(f\bigl(cn^{2}\bigr)\vartheta\bigl(cn^{2}\bigr)\bigr)^{-1}$.
The growth assumption on $f$ and the growth bound $\vartheta\left(x\right)=\mathcal{O}\bigl(x^{\nicefrac{1}{4}}\bigr)$
yields that if $j$ is large enough, then $b_{j}\leq n\leq B_{j}$
implies $cn^{2}>2^{j}$ and hence we obtain $\psi\left(i\right)\leq c^{-1}n^{-2}\bigl(f\bigl(2^{j}\bigr)\vartheta\bigl(2^{j}\bigr)\bigr)^{-1}$.
Combining these considerations, we have established that 
\[
n\psi\left(i\right)=\mathcal{O}\left(2^{-j}\left(\vartheta\bigl(2^{j}\bigr)\right)^{-\nicefrac{1}{2}}\right).
\]
Moreover, for a function $g:\mathbb{N}\rightarrow\mathbb{R}_{>0}$,
we let $E_{g}$ denote the set of $\alpha\in\left(0,1\right)$ such
that for infinitely many $j$ there is some $n$ with $b_{j}\leq n\leq B_{j}$
satisfying $\left\Vert n\alpha\right\Vert =\mathcal{O}\left(2^{-j}g\left(j\right)\right)$.
Set $h\left(j\right)\coloneqq\bigl(\vartheta\bigl(2^{j}\bigr)\bigr)^{-\nicefrac{1}{2}}$.
Applying Theorem \ref{thm: Khintchine a la Victor} with $\psi$
as in (\ref{eq: specification of the appxomation function-1}), implies
that $E_{h}$ has full measure. Therefore, for any $\alpha\in E_{h}$
we get
\begin{equation}
\left\Vert n\alpha\right\Vert \leq n\left|\alpha-\alpha_{i}\right|=\mathcal{O}\Bigl(2^{-j}\left(\vartheta\bigl(2^{j}\bigr)\right)^{-\nicefrac{1}{2}}\Bigr)\label{eq: good approximation to alpha in terms of psi-1}
\end{equation}
for infinitely many $j$. Now if $b_{j}\leq n\leq B_{j}$ for $j$
sufficiently large and $n,\alpha$ as in (\ref{eq: good approximation to alpha in terms of psi-1}),
then it follows that by taking any integer $m\leq\left(f\left(2^{j}\right)\right)^{\gamma}\bigl(\vartheta\bigl(2^{j}\bigr)\bigr)^{\frac{1}{3}}$
also the multiples 
\[
nm\leq2^{j}\left(f\left(2^{j}\right)\right)^{\gamma-1}\left(\vartheta\left(2^{j}\right)\right)^{-\nicefrac{1}{6}}
\]
satisfy that $\mathbf{1}_{\left[0,sT_{j}\right]}\left(\left\Vert \alpha(mn)\right\Vert \right)=1$
where $T_{j}=\mathcal{O}\bigl(2^{j}\left(f\left(2^{j}\right)\right)^{-\gamma}\bigr)$
is as in the Proof of Proposition \ref{prop: additive energy of good-guy-bad-guy sequence}.
If additionally $\gamma-1\geq-\beta$ holds, then we obtain that $\text{rep}_{A_{T_{j}},A_{T_{j}}}(mn)\geq\nicefrac{1}{2}2^{j}\left(f\left(2^{j}\right)\right)^{-\beta}$
holds for $j$ sufficiently large. By (\ref{eq: lower bound for counting function fo the pair correlations}),
we obtain 
\[
R\bigl(\left[-s,s\right],\alpha,T_{j}\bigr)\geq C\left(f\left(2^{j}\right)\right)^{2\gamma-\beta}\bigl(\vartheta\bigl(2^{j}\bigr)\bigr)^{\nicefrac{1}{3}}
\]
for infinitely many $j$ where $C>0$ is some constant. For the optimal
choice of the parameters $\beta,\gamma>0$, we are therefore led to
find the maximal $\beta$ such that $2\gamma-\beta\geq0$ and $\gamma-1\geq-\beta$
is satisfied. The (unique) solution is $\beta=\nicefrac{2}{3}$ and
$\gamma=\nicefrac{1}{3}$. Hence, (\ref{eq: divergence of the Pair Correlation Function})
follows for $\alpha\in E_{h}$.\\
\\
\noindent (ii) Suppose the series (\ref{eq: divergence of the reciprocal of (f(n) times n)})
converges. We keep the same sequence as in step (i) while taking $\vartheta\left(x\right)=1+\log\left(x\right)$,
as we may. The arguments of step (i) show that any $\alpha\in E_{h}$,
where $h\left(j\right)=j^{-\nicefrac{1}{2}}$, satisfies (\ref{eq: divergence of the Pair Correlation Function});
now the conclusion is that $E_{h}$ has Hausdorff dimension equal
to the reciprocal of
\[
\liminf_{x\rightarrow\infty}\frac{-\log\left(\psi\left(x\right)\right)}{\log x}=1+\liminf_{x\rightarrow\infty}\frac{\log f\left(x\right)}{\log x}.
\]
Thus, the proof is complete.

\paragraph*{Concluding remarks}

We would like to mention two open problems related to this article.
The first problem concerns extensions of Theorem \ref{thm: full measure of set of counterexamples}.
\begin{problem}
Let $\left(a_{n}\right)_{n}$ be an increasing sequence of positive
integers with $E\left(A_{N}\right)=\Omega\left(N^{3}\right)$. Has
the complement of $\NPPC\left(\left(a_{n}\right)_{n}\right)$ Hausdorff
dimension zero; or is it, in fact, empty?
\end{problem}
The second problem is related to Corollary \ref{cor: order of magnitude for the additive energy of the sequence of counter examples}.
\begin{problem}
How large has $E\left(A_{N}\right)$ to be for ensuring that $\NPPC\left(\left(a_{n}\right)_{n}\right)$
has full Lebesgue measure?
\end{problem}

\paragraph*{Acknowledgements}

Both authors would like to express their gratitude towards C. Aistleitner
for introducing us to the topic of this article, and valuable discussions.

\paragraph*{Addresses\protect \\
}

Thomas Lachmann, 

5010 Institut für Analysis und Zahlentheorie

8010 Graz,

Steyrergasse 30/II

email: lachmann@math.tugraz.at\\
\\
Niclas Technau,

5010 Institut für Analysis und Zahlentheorie

8010 Graz,

Steyrergasse 30/II,

email: technau@math.tugraz.at

\end{document}